\documentclass[12pt,reqno]{amsart}

\newcommand{\vl}{\vec \lambda}

\newcommand{\Hess}{\mbox{Hess}}
\newcommand{\wt}{\widehat}
\newcommand{\J}{J_{d_N}}
\newcommand{\N}{d_N}

\newcommand{\E}{{\mathbb E}}
\newcommand{\Q}{{\mathbb Q}}

\newcommand{\U}{{\rm U}}

\renewcommand{\phi}{\varphi}
\newcommand{\go}{\mathfrak}

\usepackage{graphicx}
\usepackage{latexsym}
\newcommand{\bcal}{\mathcal{B}}
\newcommand{\ccal}{\mathcal{C}}

\newcommand{\ecal}{\mathcal{E}}

\newcommand{\fcal}{\mathcal{F}}

\newcommand{\hcal}{\mathcal{H}}

\newcommand{\lcal}{\mathcal{L}^{DH}_{\vec \lambda}}
\newcommand{\pcal}{\mathcal{P}}
\newcommand{\mcal}{\mathcal{M}}
\newcommand{\ncal}{\mathcal{N}}
\newcommand{\qcal}{\mathcal{Q}}
\newcommand{\ocal}{\mathcal{O}}

\newcommand{\scal}{\mathcal{S}}

\numberwithin{equation}{section}

\newtheorem{maintheo}{{\sc Theorem}}
\newtheorem{mainprop}{{\sc Proposition}}
\newtheorem{mainlem}{{\sc Lemma}}
\newtheorem{maincor}{{\sc Corollary}}

\newcommand{\R}{{\mathbb R}}
\newcommand{\Z}{{\mathbb Z}}
\newcommand{\C}{{\mathbb C}}

\newcommand{\half}{{\frac{1}{2}}}

\newtheorem{theo}{{\sc Theorem}}[section]

\newtheorem{lem}[theo]{{\sc Lemma}}
\newtheorem{prop}[theo]{{\sc Proposition}}

\newenvironment{rem}{\medskip\noindent{\it Remark:\/} }{\medskip}
\newenvironment{defin}{\medskip\noindent{\it Definition:\/} }{\medskip}

\title[Random orthonormal bases of spaces of high dimension] {Random orthonormal bases of spaces of high dimension }

\author{Steve Zelditch}

\address{Department of Mathematics, Northwestern  University,
Evanston, IL 60208-2370, USA} \email{
zelditch@math.northwestern.edu}

\thanks{Research partially supported by NSF grants  \# DMS-0904252   and DMS-1206527.}

\begin{document}

\maketitle

\begin{abstract}   We consider a sequence $\hcal_N$ of finite dimensional Hilbert spaces of dimensions $d_N \to \infty$. 
Motivating examples are eigenspaces, or spaces of quasi-modes,  for a Laplace or Schr\"odinger operator on a compact
Riemannian manifold. The set of Hermitian orthonormal bases of $\hcal_N$ may be identified with $U(d_N)$, and 
a random orthonormal basis of $\bigoplus_N \hcal_N$ is a choice of a random  sequence $U_N \in U(d_N)$ from
the product of normalized Haar measures. We prove that if $d_N \to \infty$ and if $\frac{1}{d_N} Tr A |_{\hcal_N}$
tends to a unique limit state $\omega(A)$, then almost surely an orthonormal basis is quantum ergodic with
limit state $\omega(A)$. This generalizes an earlier result of the author in the case where $\hcal_N$ is the space
of spherical harmonics on $S^2$. In particular, it holds on the flat torus $\R^d/\Z^d$ if $d \geq 5$ and shows
that a highly localized orthonormal basis can be synthesized from  quantum ergodic ones  and vice-versa in relatively
small dimensions. 
\end{abstract}

       The purpose of this article is to  prove a general result on the quantum ergodicity   of random orthonormal
bases  
 $\{\psi_{N, j}\}_{j = 1}^{d_N}$ of  finite dimensional Hilbert spaces $\hcal_N \subset L^2(M)$ of dimensions $d_N \to \infty$ of a compact
Riemannian manifold $(M, g)$.
 The proof is  based on a ``moment polytope" interpretation of quantum ergodicity from \cite{Z1}: 
the  quantum variances of a Hermitian  observable $A \in \Psi^0(M)$ are identified  with   moments of inertia  of the
convex  polytopes $
\pcal_{\vec \lambda}$ defined as the convex hull of the vectors   $\vec \lambda = (\lambda_{ 1}, \dots, \lambda_{ d_N})$ of eigenvalues (in all possible orders)  of 
 $\Pi_N
A \Pi_N$ where $\Pi_N: L^2(M) \to \hcal_N$ is the orthogonal projection. Equivalently, 
$\pcal_{\vec \lambda}$ is  the image of the coadjoint orbit $\ocal_{\vec \lambda}$ of the diagonal matrix
$D(\vec \lambda)$  under the moment map
for the   Hamiltonian   action of the maximal torus $T_{d_N} \subset U(d_N)$ of diagonal matrices acting by conjugation   on 
$\ocal_{\vec \lambda}$. In particular, the  main estimates of quantum 
ergodicity can be formulated in terms of estimates of the  first four moments of inertia of  $\pcal_{\vec \lambda}$. The main result, Theorem \ref{MAINTHEO},  states that random
orthonormal bases are almost surely quantum ergodic as long as $d_N \to \infty$ and  $\frac{1}{d_N} Tr \Pi_N A \Pi_N \to \omega(A)$
for all $A \in \Psi^0(M)$, where $\omega(A)$ is the Liouville state.  More generally, if these traces have any unique limit
state, then almost surely it is the quantum  limit of a random orthonormal basis.  The proof is essentially implicit in
\cite{Z1}, but we bring it out explicitly here and also give detailed calculations of the moments of intertia, which
seem of independent interest.

Quantum ergodicity of random orthonormal bases is a rigorous result on  the `random wave model' in quantum chaos, according
to which eigenfunctions of quantum chaotic systems should behave like  random waves. It also has implications for
the approximation of modes by quasi-modes. Since eigenfunctions of the Laplacian $\Delta$ of a compact
Riemannian manifold $(M, g)$ form an orthonormal basis, it is natural to compare the orthonormal basis of eigenfunctions to a `random orthonormal basis'. 
In  \cite{Z1}, the result of this article was proved for the special case where  $\hcal_N$  is the space
of degree N spherical harmonics on the standard $S^2$. In \cite{Z2} the
quantum ergodic property was generalized to any compact Riemannian manifold, with $\hcal_N$ the span
of the eigenfunctions in a spectral interval $[N,  N + 1]$ for $\sqrt{\Delta}$.  Related results have recently been proved in \cite{SZ,BL}. The dimension of such  $\hcal_N$ grows at the rate $N^{m-1}$ where $m = \dim M$ and thus a random 
element of $\hcal_N$ is a superposition of $N^{m-1}$ states.
The  results of this article  show that the same quantum ergodicity property holds for sequences of eigenspaces (or linear combinations)
whose dimensions $d_N $ tend to infinity  at any rate. 
For instance, 
the results show that random orthonormal bases of eigenfunctions  on a flat torus of dimension $\geq 5$  are quantum ergodic (for the
precise statement, see \S \ref{TORUS}, and for further discussion, see \S \ref{DISC}.)

 To explain the moment map interpretation and the variance formula, recall that quantum ergodicity is concerned
with quantum variances, i.e. with  the dispersion from the mean of the diagonal part of a Hermitian matrix $H_N$ on a large dimensional
vector space $\hcal_N$.  The matrix $H_N$ is the restriction \begin{equation} \label{TAN}  T_N^A : = \Pi_N A \Pi_N \end{equation}  to $\hcal_N$   
of a pseudo-differential operator
$A \in \Psi^0(M)$; here $\Pi_N$ is the orthogonal projection to $\hcal_N$ and $\Psi^0(M)$ is the space of pseudo-differential operators
of order zero.  The same methods and results apply to other context such as semi-classical pseudo-differential operators or to Toeplitz
operators on holomorphic sections of powers of a positive line bundle \cite{SZ}. 
Given an ONB $\{\psi_{N, j}\}_{j = 1}^{d_N}$ of $\hcal_N$ we define
the quantum variances of the ONB  (indexed by $A \in \Psi^0(M)$ by
\begin{equation} \label{VARDEF} V_A(\{\psi_{N k}\}) : = \frac{1}{d_N} \sum_{j = 1}^{d_N} |\langle A \psi_{N, j},  \psi_{N, j} \rangle - \omega(A)|^2.  \end{equation}
Here, $\omega(A) = \int_{S^* M} \sigma_A d\mu_L$ where $d\mu_L$ is normalized Liouville measure (of
mass one).  

\begin{defin} A sequence $\{\psi_{N j}\}_N$ of  ONB's of $\hcal_N$ is   a quantum ergodic ONB of $L^2(M)$
if 
\begin{equation} \label{EP}  (\ecal \pcal)\;\;\;\;\;\;\; \lim_{N \to \infty} V_A(\{\psi_{N k}\}) = 0,\;\;\; \forall A \in \Psi^0(M).  \end{equation}
\end{defin}
By a standard diagonal argument, this implies that almost all the individual elements $\langle A \psi_{N, j}, \psi_{N, j} \rangle $
tend to $\omega(A)$. Since this aspect of quantum ergodicity is the same as in \cite{Z1,SZ} (e.g.) we do not discuss it here.

 To define  random orthonormal bases, we introduce the
probability space $(\ocal \ncal \bcal ,d\nu)$, where $\ocal \ncal \bcal$ is the infinite
product of the sets $\ocal \ncal \bcal_N$ of orthonormal bases of the spaces
$\hcal_N$, and $\nu =\prod_{N=1}^{\infty} \nu_N$, where $\nu_N$ is
Haar probability measure on $\ocal \ncal \bcal_N$. A point of $\ocal \ncal \bcal$ is thus
a sequence ${\bf \Psi} = \{(\psi^N_1, \dots, \psi^N_{d_N}) \}_{N\geq 1}$ of
orthonormal basis. 
Given one orthonormal basis $\{e_j^N\}$  of $\hcal_N$ any other is  related to it by a unique unitary
matrix. So the probability space is equivalent to the product
\begin{equation} \label{ONBSPACE} (\ocal \ncal \bcal ,d\nu) \simeq \prod_{N = 1}^{\infty} (U(d_N), dU) \end{equation}
where $dU $ is the unit mass Haar measure on $U(d_N)$.   Here we are working with Hermitian orthonormal
bases and Hermitian pseudo-differential operators. We could also work with real self-adjoint operators and real
orthornormal bases, which are then related by the orthogonal group. The results in that setting are essentially the same
but the proofs are somewhat more complicated; for expository simplicity we stick to the unitary Hermitian framework. 

Let $A \in \Psi^0$ and denote the eigenvalues of $T^A_N$
by $\lambda_1,\dots,\lambda_{d_N}$.  The  empirical measure of eigenvalues of $T^A_N$ is defined by
 \begin{equation} \label{EMP} \nu_{\vec \lambda_N} := \frac{1}{\N}  \sum_{j = 1}^{\N} \delta_{\lambda_j} . \end{equation}
Its moments are given by  
\begin{equation} \label{POWER} p_k(\lambda_1,\dots,\lambda_{d_N})=\sum_{j=1}^{d_N} \lambda_j^k = {\rm Tr}  (T_N^A)^k. 
\end{equation}
To obtain quantum ergodicity, we  we put the following constraint on the sequence $\{\hcal_N\}$:

\begin{defin}We say that $\hcal_N$ has local Weyl  asymptotics if, for all $A \in \Psi^0(M)$,
\begin{equation}  \label{TRACE}  \frac{1}{d_N} \rm{Tr} T_N^A = \omega(A) + o(1). \end{equation}

\end{defin}

In fact, the results generalize to the case where $\omega(A)$ is replaced by any other limit state, i.e. $\int_{S^*M} \sigma_A d\mu$
where $d\mu$ is another invariant probability measure for the geodesic flow.

Our main result is:

\begin{maintheo} \label{MAINTHEO} Let $\hcal_N$ be a sequence of subspaces of $L^2(M)$ of dimensions $d_N = \dim \hcal_N
\to \infty$. Assume that $\frac{1}{d_N} Tr \Pi_N A \Pi_N = \omega(A) + o(1) $ for all $A \in \Psi^0(M)$. 
Then with probability one in  $(\ocal \ncal \bcal ,d\nu) $, a random orthonormal basis of $\bigoplus_N \hcal_N$ is quantum ergodic. \end{maintheo}


A natural question (which we do not study here) is whether a random orthonormal basis is QUE, i.e. whether
$$\max \{|\langle A \psi_{N, j},  \psi_{N, j} \rangle - \omega(A)|^2, \;\;\; j = 1, \dots, d_N\} \to 0\;\; (a.s.) d\nu \;\;?. $$
As a tail event the probability of a random orthonormal basis being QUE is either 0 or 1. 

We now explain how to formula Theorem \ref{MAINTHEO}  in terms of moment maps and polytopes.
Quantum 
ergodicity of  a random orthonormal bases concerns the dispersion from the mean of the diagonal part of $T_N^A$. 
The diagonal part depends on the choice of an orthonormal basis  of $\hcal_N$. Once an orthonormal basis is fixed,
$iT^A_N$ can be identified with an element $H_N$ of the Lie
algebra ${\go u}(d_N)$ of
$\U(d_N)$, and a unitary
change of the orthonormal basis results in the conjugation $H_N \to U_N^* H_N U_N$ of $H_N$. If the vector
of eigenvalues of $H_N$ is denoted $\vec \lambda_N$, then the conjugates sweep out the orbit $\ocal_{\vec \lambda_N}$.
 Let ${\go t}(d_N)$ denote the Cartan subalgebra of diagonal
elements in ${\go u}(d_N)$, and let $\|\cdot\|^2$ denote
the Euclidean inner product on ${\go t}(d_N)$.  Also let
$$J_{d_N}:i{\go u}(d_N)\rightarrow i{\go t}(d_N)$$
denote the orthogonal projection (extracting the diagonal). Extracting the diagonal from each element of the orbit is precisely the moment map
\begin{equation} \label{J} \J: \ocal_{\vec \lambda_N} \to \pcal_{\vec \lambda_N}  \subset i {\go t}(d_N),\;\;\;
\J (U D(\vec \lambda) U^*) = \left(\dots, \sum_{j = 1}^{d_N} \lambda_j |U_{i j}|^2, \dots \right) \end{equation}  of the conjugation action of $T_{\N}$.
Finally, let $$\bar J_{d_N} (H)=\left(\frac{1}{d_N}{\rm Tr}\;H\right){\rm
Id}_{d_N}, \;\;\;\; D_0(\vec \lambda_N) = D(\vec \lambda_N) - \left(\frac{1}{d_N}{\rm Tr}\;H\right){\rm
Id}_{d_N}\;,$$ for Hermitian matrices $H\in i{\go u}(d_N)$. We also introduce  notation for the diagonal
of $D_0(\vec \lambda)$:
\begin{equation} \label{CAPL}  D_0(\vec \lambda) = D (\vec \Lambda), \;\; \mbox{with};\;
\Lambda_j : = \lambda_j -\frac{1}{d_N} \sum_{j = 1}^{d_N} \lambda_j. \end{equation}
Thus, $$H=H^0 +\bar J_d(H), \;\;\; \mbox{resp.}\;\;\;D(\vec \lambda_N) = D_0(\vec \lambda_N) + \left(\frac{1}{d_N}{\rm Tr}\;H\right){\rm
Id}_{d_N}$$   with $H^0$ traceless, corresponds to the
decomposition  ${\go u} (d_N) =  {\go su} (d_N)\oplus \R$.

As this description indicates,  quantum ergodicity of random orthonormal bases is mainly a result about the
asymptotic geometry of the polytopes $\pcal_{\vec \lambda_N}$ corresponding to a sequence $T_N^A$ of Toeplitz operators.
The pushforward of the $U(d_N)$-invariant normalized measure on $\ocal_{\vec \lambda}$ to $\pcal_{\vec \lambda}$ is
the so-called Duistermaat-Heckman measure $d\lcal$, a piecewise polynomial measure on $\pcal_{\vec
\lambda}$. 
To prove almost sure quantum ergodicity, we  prove that
  for all such sequences $T_N^A$ and their spectra $\{\vec \lambda_N\}$, the second and fourth moments of inertia
of $\pcal_{\vec \lambda}$ with respect to $d \lcal$ are bounded. We use the property in Definition \ref{TRACE}
to replace $\omega(A)$ by the centers of mass, i.e. the scalar matrix with the same trace as $T_N^A$. The Kolmogorov
strong law of large numbers then gives the quantum ergodicity property. 
 In \cite{Z3}, we study higher moments and their implication for the limit shape of $\pcal_{\vec \lambda}$ along
a sequence $\{\lambda_N\}$ with a limit empirical measure.


We asymptotically evalute the moments using the  Fourier transform
\begin{equation} \label{FT} \hat{\mu}_{\vec \lambda}(X) := \int_{\ocal(\vec \lambda)} e^{i \langle X, \mbox{diag}(Y)\rangle} d \mu_{\vec \lambda}(Y) \end{equation} 
of the $\delta$-function on $\ocal_{\vec \lambda}$.  Here,  we assume $X \in \R^{d_N}$.  We may identify $X$  with
a diagonal matrix,  and then  $\langle X, \mbox{diag}(Y)\rangle = Tr X Y$, and we get the standard Fourier transform. We obviously have:

\begin{mainlem} \label{VARFORM} Let $\Delta$ be the Euclidean Laplacian of $\R^{d_N}$ acting in the $X$ variable. Then,

$$ \left\{ \begin{array}{l}   m_{2}(\pcal_{\vec \lambda_N}) : = \E|| \J (U^* D(\vec \lambda) U) ||^2  = - \Delta \hat{\mu}_{\vec \lambda}(X) |_{X = 0}, \\ \\  m_{4}(\pcal_{\vec \lambda_N})  : =
\E||\J (U^* D(\vl) U) ||^4 =  \Delta^2 \hat{\mu}_{\vec \lambda}(X) |_{X = 0},  \end{array} \right.$$

\end{mainlem}

 We translate $\vec \lambda$ by its center of mass to make the center of
mass of $\pcal_{\vec \lambda}$ equal to $0$, i.e. $\sum \lambda_j = 0$. 
 Using a formula for $ \hat{\mu}_{\lambda}(X)$  in terms of Schur polynomials,  we  prove

\begin{mainlem} \label{4}  Let $p_k$ be the  power functions \eqref{POWER}. Assume that $p_1(\vec \lambda) = 0$.
Then,
$$  \left\{ \begin{array}{l} \Delta \hat{\mu}_{\vec \lambda}(0) = \frac{p_2(\vec\lambda)}{d_N+1},  \\ \\
\Delta^2 \hat{\mu}_{\vec \lambda}(0) 
= \beta_4(d_N)\;\; p_2^2 (\vec \lambda), \\ \\
\mbox{with}\;\;\; \beta_4(d_N) =  \left(\frac{ 4 d_N(d_N -1) }{(\N + 1) \N^2 (\N-1) }
-   \frac{ 4 d_N(d_N -1) }{(\N + 2)(\N+1) \N  (\N - 2)} 
 +    \frac{ (12 \N^2 + 4 d_N (d_N -1)  ) }{(\N + 3) (\N+2) (\N+1)\N}
\right) 
. \end{array}  \right. $$
\end{mainlem}

The proof of Theorem \ref{MAINTHEO} follows directly from Lemma \ref{4} and the Kolmogorov SLLN (strong law
of large numbers). When  $d_N$ grows fast enough it also follows directly from the Borel-Cantelli Lemma.
We first introduce notation for the basic random variables:

\begin{defin}\label{YNDEF}
$$\left\{\begin{array}{l}
Y^A_N: \ocal \ncal \bcal_N \to [0,+\infty), \;\;\; \Psi = (U_{d_1}, U_{d_2}, \dots)\\[6pt]
Y^A_N({\bf \Psi}) := \|\J(U_N^* D(\vl) U_N) -
\overline{D(\vl) } \|^2 =  \|\J(U_N^* D_0(\vl) U_N)  \|^2 \end{array} \right.$$
\end{defin}
 Then Lemma \ref{VARFORM}-Lemma \ref{4}   determine the asymptotics of their mean and  variance
$$\mbox{\rm Var}\left(Y^A_N\right) := \E ((Y_N^{A})^2) - (\E ((Y_N^{A}))^2. $$
\begin{maincor} \label{VARCOR} 
$$\begin{array}{lll} \mbox{\rm Var}\left(Y^A_N\right)  & = & \left(\beta_4(d_N)    -   \frac{1}{(d_N+1)^2}
 \right)  p_2^2 ( \vec \Lambda) \simeq \frac{3}{d_N^2} p_2^2 (\vec \Lambda) .\end{array}$$
\end{maincor}

 The Lemma first implies that
$  \E|| \J (U^* D(\vec \lambda) U) ||^2  $ is bounded for all $A \in \Psi^0(M)$. Hence,  $  \E (\frac{1}{d_N} || \J (U^* D(\vec \lambda) U) ||^2)
\to 0 $ as long as $d_N \to \infty$.  Thus, the mean of the quantum variances
\eqref{VARDEF} tends to zero. As in \cite{Z1,SZ} we then apply the 
Kolmogorov SLLN (or the martingale convergence theorem). The $\{Y_N^A\}$
is a sequence of   independent
random variables as $N$ varies and Lemma \ref{4} shows that they have bounded variance. Hence the  SLLN implies
that the partial sums,
\begin{equation} \label{SN} S_N : = \sum_{n \leq N}  \frac{1}{d_n} (Y_n^A - \E Y_n^A) \end{equation}
have the property,
\begin{equation} \frac{1}{N} S_N  \to 0, \;\; \mbox{almost surely} \end{equation}
and this is equivalent to quantum ergodicity of random orthonormal bases. 
As mentioned above,  if $d_N$ grows at a faster rate one can obtain stronger results from the Borel-Cantelli
Lemma:  E.g. if $\sum_{n = 1}^{\infty} \frac{1}{d_n} < \infty$,
one obtains almost sure convergence $\frac{1}{d_n} Y_n^A \to 0$ (a.s.). 

Since the argument above only requires that $\frac{1}{d_n} \E Y_n^A \to 0$ and $\mbox{Var}(Y_n^A)$ is bounded,
it does not require any assumption that $\E Y_n^A $ tends to a limit.  Our calculations therefore go beyond what
is necessary for almost sure quantum ergodicity, and  pertain to the asymptotic geometry of the
polytopes $\pcal_{\vec \lambda}$.   There is a natural  condition
on the this sequence of polytopes:

\begin{defin} \label{LIMIT} We say that the sequence $\{\hcal_N\}$ has Szeg\"o asymptotics if,  for all $A \in \Psi^0(M)$, there exists
a unique  weak* limit, 
$ \nu_{\vec \lambda_N}  \to \nu_A \in \mcal(\R)$     as $N \to \infty$. Here, $\mcal(\R)$ is the set of probability
measures on $\R$.  
\end{defin}

Under this stronger assumption,  Lemma \ref{4} gives moment asymptotics:

\begin{mainprop} \label{1} Let $\vec \lambda_N \in \R^{d_N}$ be a sequence of vectors with the property that the 
empirical measures \eqref{EMP} tend to a weak limit $\nu$. Then $$\left\{ \begin{array}{l} 
 m_{2}(\pcal_{\vec \lambda_N}) \to \int_{\R} (t - \bar{t})^2 d\nu, \\ \\
 m_{4}(\pcal_{\vec \lambda_N}) \to  4  \left(\int_{\R} (t - \bar{t})^2 d\nu \right)^2
 \end{array} \right.$$

\end{mainprop}

This Proposition is closely related to the ``Weingarten theorem" that the matrix elements $\sqrt{d_N} U_{ij} $ are asymptotically
complex normal random variables, where $U_{ij}$ are the matrix elements of $U \in U(d_N)$ \cite{W}. Perhaps
this explains why the fourth moment is a constant multiple of the square of the second moment. It would be interesting
to see if the pattern continues; we plan to study $\pcal_{\vec \lambda}$ further in \cite{Z3}.

\subsection{\label{DISC} Discussion}

The motivation for proving  quantum ergodicity of random orthonormal bases for $\hcal_N$ of any dimensions tending
to infinity 
was prompted by  the general  question: how many diffuse states (modes or quasi-modes) does it take to synthesize localized 
modes or quasi-modes?
Vice-versa, how many localized states does it take to synthesize diffuse states?  We would like to synthesize
entire orthonormal bases rather than individual states and measure the dimensions of the space of states in terms of the Planck constant $\hbar$. Let us consider some examples.

In the case of the standard $S^2$, the eigenspaces $\hcal_N$ of $\Delta$ are the spaces of spherical harmonics of degree $N$.
They have the well-known highly localized basis $Y^N_m$ of joint eigenfunctions of $\Delta$ and of rotations around
the $x_3$-axis. By localized we mean that a sequence $\{Y^N_m\}$ with $m/N \to \alpha$ microlocally concentrates
on the  invariant tori in $S^* S^2$ where $p_{\theta} = \alpha$. Here, $p_{\theta}(x, \xi) = \xi(\frac{\partial}{\partial \theta})$
where $\frac{\partial}{\partial \theta}$ generates the $x_3$-axis rotations. On the other hand, it is proved in \cite{Z1}
that independent  ``random" orthonormal bases of $\hcal_N$ are quantum ergodic, i.e. are highly diffuse in $S^* S^2$. 
Since $\dim \hcal_N = 2N + 1$, it is perhaps not surprising that the same eigenspace can have both highly localized
and highly diffuse orthonormal bases when its dimenson is so large. The question is, how large must it be for
such incoherently related bases to exist?

A setting where the eigenvalues have high multiplicity but of a lower order of magnitude than on $S^2$ is 
that of flat rational tori  $\R^n/L$ such as
$\R^n/\Z^n$. Of course it has an orthonormal basis of localized eigenfunctions, $e^{i \langle k, x \rangle}$. 
 The key feature of such rational tori is the high multiplicity of eigenvalues of the
Laplacian $\Delta$ of the flat metric. It is well-known and easy to see that the multiplicity is the
number of lattice points of the dual lattice $L^*$ lying on the surface of a Euclidean sphere. 
 We  denote
 the distinct multiple $\Delta$-eigenvalues by $\mu_N$, the corresponding eigenspace by $\hcal_N$
 and  the multiplicity of $\mu_N^2$ by $d_N = \dim \hcal_N$.
In dimensions
$n \geq 5$, $d_N \sim \mu_N^{n - 2}$,  one degree lower than  the maximum possible multiplicity of
a $\Delta$-eigenvalue on any compact Riemannian manifold,  achieved on the standard $S^n$. Further, 
$\frac{1}{d_N} Tr \Pi_N A \Pi_N \to \omega(A)$. 
Hence, the  results of this article show that despite the relatively slow growth of $d_N$ on a flat rational
torus, orthonormal bases of $\hcal_N$ in dimensions $\geq 5$ are almost surely  quantum ergodic. The statement for dimensions $2, 3, 4$ is 
more complicated (see \S \ref{TORUS}).

An interesting setting where the behavior of eigenfunctions is largely unknown is that of KAM systems. For these,
one may construct a  `nearly' complete and  orthonormal basis for $L^2(M)$ by highly localized quasi-modes
associated to the Cantor set of invariant tori. It seems unlikely that the actual eigenfunctions are  quantum ergodic;  
but the results of this article show that if they resemble random combinations of the quasi-mode, then it is
possible that they are.    Further discussion is in \S \ref{QM}.

\section{Background}

In this section, we review the definition of random orthonormal basis and relate it to properties of the moment
map for the diagonal action of the maximal torus $T_{\N}$ on co-adjoint orbits of $U(\N)$.

\subsection{Random orthonormal bases of eigenspaces} Suppose that we have a sequence of Hilbert spaces $\hcal_N$
$N = 1, 2, \dots$ of dimensions $d_N = \dim \hcal_N \to \infty$. We define  the large Hilbert space
$$ \hcal = \bigoplus_{N = 1}^{\infty} \hcal_N $$
and orthogonal projections
\begin{equation} \label{PIN} \Pi_N : \hcal \to \hcal_N . \end{equation} 
We then consider the orthonormal bases \eqref{ONBSPACE} of $\hcal$ which arise from sequences 
of orthonormal bases of $\hcal_N$.



\subsection{The basic random variables}

Let $A \in \Psi^0(M)$ be a zeroth order pseudo-differential operator. By a Toeplitz operator we mean
the compression $T_N^A$  \eqref{TAN}  of $A$ to $\hcal_N$.

Given one ONB of $\hcal_N$,  $T^A_N$   can be identified
with a Hermitian $d_N\times d_N$ matrix. 
We fix orthonormal bases $\{e^N_j\}_{j = 1}^{d_N}$ of $\hcal_N$ and  introduce the random variables:
  \begin{equation} A_{Nj}({\bf \Psi}) = \left| \langle A  \psi^N_j,
\psi^N_j \rangle- \omega(A) \right|^2 = \left|(T^A_N \psi^N_j, \psi^N_j)- \omega(A)
\right|^2 = \left|(U^*_N T^A_N U_N e^N_j, e^N_j)- \omega(A)
\right|^2,\end{equation}
where ${\bf \Psi} = \{U_N\},\ U_N\in \U(d_N)\equiv \ocal \bcal \ncal_N$.
We also define
  \begin{equation}\label{Anj}\wt
A_{Nj}({\bf \Psi})=\left|(U^*_N T^A_N U_N e^N_j,
e^N_j)-\frac{1}{d_N}{\rm Tr}\;T^A_N \right|^2 \;.\end{equation}


Evidently,

\begin{equation}\label{Y}\frac{1}{d_N}
Y^A_N({\bf \Psi}) =\frac{1}{d_N}\sum_{j=1}^{d_N}\wt A_{Nj}({\bf \Psi})
=\frac{1}{d_N}\sum_{j=1}^{d_N}A_{Nj}({\bf \Psi})
+ o(1)\end{equation} (where the $o(1)$ term is independent
of $ {\bf \Psi}).$  Thus,

\begin{lem}  \cite{Z1,SZ} The ergodic property of an ONB ${\bf \Psi}$  $(\ecal \pcal)$ is equivalent to:
\begin{equation}\label{EP*}
\lim_{N\rightarrow\infty}\frac{1}{N}\sum_{n=1}^N
\frac{1}{d_n} Y^A_n ({\bf \Psi}) = 0\;,\quad \forall A \in
\Psi^0 (M)\;. \end{equation} 

\end{lem}

As mentioned in the introduction, it follows by a 
standard diagonal argument  that almost all the individual elements $\langle A \psi_{N, j}, \psi_{N, j} \rangle $
tend to $\omega(A)$ for all $A$. We do not discuss this step since it is nothing new.




\subsection{Moment map interpretation}

In the case where the components of $\vec \lambda_N$ are distinct,
the covex polytope $\pcal_{\vec \lambda_N}$  is the permutahedron determined
by $\lambda$, that is, the simple convex polytope defined as the convex hull of the 
points $\{\sigma(\vec \lambda_N)\}$ where $\sigma \in S_N$ runs over the symmetric
group on $d_N$ letters (i.e. the Weyl group of $U(d_N)$). The center of mass is the unique
point $X \in \pcal_{\vec \lambda_N}$ so that 
$$\sum_{\sigma \in S_{d_N}} \overline{X \sigma(\vec \lambda_N)} = 0 \iff X = \frac{1}{(d_N)!} \sum_{\sigma \in S_{d_N}}\sigma (\vec \lambda_N)$$
where $\overline{XY} = X - Y$ is the vector from $X$ to $Y$.  The center of mass is evidently invariant under
$S_{d_N}$, hence has the form $(a, a, \dots, a)$ for some $a$ and clearly $a = \frac{1}{d_N} \sum_{j = 1}^{d_N} \lambda_j.$

In 
effect, we want to asymptotically calculate the moments of inertia of the sequence of
permutahedra associated to a Toeplitz operator.  
\vspace{1cm}

\begin{center}
\includegraphics[scale=0.5]{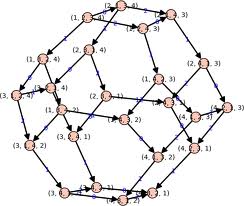} \text{Permutahedron}
\end{center}
\vspace{1cm}


\subsection{Symmetric polynomials and Schur polynomials}

The elementary symmetric polynomials of $N$ variables are defined by
$$e_k(X_1, \dots, X_N) = \sum_{i_1 < i_2 < \cdots < i_k \leq N} x_{i_1} \cdots x_{i_N}. $$
If one replaces $<$ by $\leq $ one obtains the complete symmetric polynomials $h_k$. The 
 Schur polynomials are symmetric polynomials defined by 
$$S_{\lambda} = \det \begin{pmatrix} h_{\lambda_i + j - i} \end{pmatrix} = \det \begin{pmatrix} e_{\mu_i + j - i} \end{pmatrix} 
$$
where $\mu$ is a dual partition to $\lambda$.



\subsection{Fourier transform of the orbit}

We can compute the moments using the Fourier transform \eqref{FT} of the orbital measure on
the orbit of $D(\vec \lambda)$.

An  explicit formuae for $ \hat{\mu}_{\lambda}(X) $ is given in the first line of the proof of  Theorem 5.1 of \cite{OV}:

\begin{lem} \label{OV} For $U(d)$, 
\begin{equation} 
\hat{\mu}_{\vec \lambda}(X) = (d-1)! \cdots 0! \sum_{\mu: \ell(\mu) \leq d}
\frac{S_{\mu}(X) S_{\mu} (i \vec \lambda)}{(\mu_1 + d-1)! (\mu_2 + d-2)! \cdots \mu_N!}. \end{equation}
Here, $\ell(\mu)$ is the number of rows of the partition $\mu$.  The degree of $s_{\mu}$ is $|\mu|$,
the number of boxes. 
\end{lem}

Since we would like to shift the center of mass of $\pcal_{\vec \lambda}$ to the origin, we mainly consider
$ \hat{\mu}_{\vec \Lambda}(X)$ the Fourier transform of the traceless orbit  (see \eqref{CAPL}).



\section{Proof of Proposition \ref{1}: Moment asymptotics}

\subsection{Second moment asymptotics}

We now prove:

\begin{mainlem} \label{2m} \cite{Z1,Z2,SZ} Let $\vec\lambda=(\lambda_1,\dots,\lambda_{d_N})\in\R^{d_N}$, and let
$D_0(\vec\lambda_N)$ denote the trace zero diagonal matrix with entries \eqref{CAPL}. Thus, $p_1(\vec \Lambda) = 0$
\eqref{CAPL}. Then

\begin{equation}\label{orbit-int}  \E Y_N^{A} = \int_{\U(d)} \|\J (U^*D_0(\vec\lambda) U)\|^2 dU =
\frac{p_2(\vec\Lambda)}{d_N+1}, \;,\end{equation}
where  as above,  $d U$ is the normalized  Haar probability measure on $\U(d_N)$.
\end{mainlem}

This Lemma was proved in \cite{Z1,Z2, SZ} using the so-called Itzykson-Zuber-Harish-Chandra formua
for the  Fourier transform of the orbit, and  again using Gaussian integrals. The proof we give here 
generalizes   better to higher moments. We also sketch a proof using the Weingarten formulae.

\begin{proof}

We   use Lemma \ref{OV} to obtain
$$ \E||\J(U^* D(\vec \lambda) U) ||^2  =  (\N-1)! \cdots 0! \sum_{\mu: |\mu| = 2, \ell(\mu) \leq \N}
\frac{\Delta S_{\mu}(0) S_{\mu} (i \vec \lambda)}{(\mu_1 +\N-1)! (\mu_2 + \N-2)! \cdots \mu_{\N}!}. $$
We sum over the Young diagrams with exactly two boxes and $\leq \N$ rows. There are just two of them: one row
of two boxes or two rows of one box each corresponding respectively to the Schur functions $S_{(2, 0)}$, $S_{(1,1)}$. 
Note that $S_{1^k} = e_k$ is the kth elementary  symmetric function and $S_{(k)} = h_k$ is the complete kth degree
symmetric function.

We then  translate $\vec \lambda$ to $\vec \Lambda$ so that $\sum_j \Lambda_j = e_1(\vec \Lambda) = 0$, i.e. we replace $D(\vec \lambda_N)$
by $D_0(\vec \lambda_N)$.


Since the degree $|\mu| = 2$,  then we can only use $\mu = (2), (1 1)$ and
$$S_{(1,1)} = e_2 =  \sum_{i < j} x_i x_j, \;\;\; S_{(2 0)} = e_1^2 - e_2 = \sum_j x_j^2 + \sum_{i < j} x_i x_j. $$
But
$$\Delta e_2 \equiv 0, \;\;\; \Delta (e_1^2 - e_2) = 2 ||\nabla e_1||^2 = 2 d_N. $$

For each monomial $X_i X_j$ we have $\Delta X_i X_j = 2 \delta_{ij}. $ Thus, $\Delta S_{(1,1)}  = 0$ and
$\Delta S_{(2 0)} = 2 \N.$ Since the Schur polynomials are homogeneous of degree 2, we can remove the $i$
under the Schur polynomials to get an overall factor of $-1$, which is cancelled by the $-$ sign from $\Delta$. Thus,
$$\begin{array}{lll}  \E||\J (U^* D_0(\vl) U) ||^2  & = & (2 \N)  (\N-1)! 
\frac{ S_{(2,0)} (i \vec \Lambda)}{( \N + 1)!} = \frac{(2 \N)}{(\N + 1) \N}  S_{(2, 0)}(i \vec \Lambda) 
\\ &&\\ & = & \frac{2}{\N + 1} S_{(2, 0)}(\vec \Lambda_N) = \frac{2}{\N + 1} ( e_1^2 - e_2)(\vec \Lambda).  \end{array}$$
Since
$e_1(\vec \Lambda_N) = 0$ we find that 
$$\begin{array}{lll}  \E||\J(U^* D_0(\vl) U) ||^2  & = & -  \frac{2}{\N + 1}  e_2(\vec \Lambda) = \frac{1}{\N+ 1}
p_2(\vec \Lambda_N). \end{array}$$
Here we use that
$$e_1 = p_1, \; 2 e_2 = e_1 p_1 - p_2. $$
The formula agrees with the one stated in the Lemma \ref{2m}.

\end{proof}

\subsection{Weingarten formulae for the expectation}

As a second proof, we use the Weingarten formula for integrals of polynomials over $U(N)$ \cite{W}. We denote
the eigenvalues of $ D_0(\vec \lambda) $ by $\vec \Lambda$. Then,

$$\begin{array}{lll} || \mbox{diag} (U^* D_0(\vec \lambda)  U)||^2  & = &   \sum_{ j_1,  j_2} \Lambda_{j_1}    \Lambda_{j_2} 
 \sum_{i}  |U_{i j_1}|^2 |U_{i j_2}|^2.
\end{array}$$
The Weingarten formulae for these special polynomials state that asymptotically $\sqrt{d_N} |U_{ij}|^2$ is a
complex Gaussian random variable of mean zero and variance one. Thus, to leading order,
$$\begin{array}{l} \int_{U(\N)} |U_{i_1 j_1}|^2 | U_{i_1 j_2} |^2  dU 
\simeq  \N^{-2} (1 + 
 \delta_{j_1 j_2} ),  \end{array}$$ and 
\begin{equation} \label{EQ} \begin{array}{lll} \sum_{j_1, j_2} \Lambda_{j_1} \Lambda_{j_2} 
\sum_{i_1} \int_{U(\N)} |U_{i_1 j_1}|^2 | U_{i_1 j_2} |^2  dU 
& \simeq & \N^{-1}  ( 2  \sum_j \Lambda_j^2
+  \sum_{j_1 \not= j_2} \Lambda_{j_1} \Lambda_{j_2} ). \end{array}  \end{equation}

Since $$0 = (\sum \Lambda_j)^2 =  \sum \Lambda_j^2 +  \sum_{j \not= k} \Lambda_j \Lambda_k$$
we get 
$$\eqref{EQ} \simeq   \N^{-1}    \sum_j \Lambda_j^2.  $$

\subsection{Proof of Proposition \ref{1}: Variance and fourth moment asymptotics}

We now prove the 4th moment identity in Proposition \ref{1}, which is the main new step in this article.



To calculate the variance of $Y_N^A$ we use the expression in  Lemma \ref{VARFORM} in terms of $\hat{\mu}_{\vec \lambda}$
and then use the formula of Lemma \ref{OV}.

 A Schur polynomial $S_{n_1, \dots, n_d}(x_1, \dots, x_d)$ of degree n  in $d$ variables is parameterized by a
 partition of of the degree $n= n_1 + n_2 + \cdots + n_d$ into $d$ parts. When $n = 4$ and $d \geq 4$ there
are 5 partitions: 
\begin{equation} \label{LIST}
\left\{ \begin{array}{l}  S_{1,1,1,1} (x) = e_4 =  \sum_{1 \leq  i < j < k < \ell} x_i x_j x_k x_{\ell}; \\ \\ S_{2,1,1}(x_1, \dots, x_N) = e_1 e_3\\ \\
S_{2, 2, 0} = e_2^2 - e_1 e_3 
\\ \\
S_{4, 0,0} = e_1^4 - 3 e_1^2 e_2 + 2 e_1 e_3 + e_2^2\\ \\
S_{3, 1, 0} = e_1^2 e_2 - e_2^2 - e_1 e_3. \end{array} \right. \end{equation}
 We note 
that $\Delta e_k(X) = 0$ for all $k$, so $\Delta e_k e_n = 2 \nabla e_k\cdot \nabla e_n $. Also,
$\nabla e_1$ is a constant vector. So
$\Delta^2 e_1 e_3 = \nabla e_1 \cdot \nabla \Delta e_3 = 0$ and $$\Delta^2 e_1^2 e_2 = 
4 \Delta (e_1 \nabla e_1 \cdot \nabla e_2) = 8 \nabla e_1 \cdot \nabla (\nabla e_1 \cdot \nabla e_2) =
8 \rm{Tr} \Hess e_2 = 0. $$ Here, $\Hess$ denotes the Hessian.
We also use that $\Delta (\nabla f \cdot \nabla g) = 2 \Hess(f) \cdot \Hess(g)$ when $\Delta f = \Delta g = 0$.
Also,
$$ \nabla e_1 \cdot \nabla (\nabla e_1 \cdot \nabla e_2)  = (1, 1, \dots, 1) \cdot \sum_{j, k} \frac{\partial^2 e_2}{\partial
x_j \partial x_k} \frac{\partial}{\partial x_k} = \mbox{Tr Hess} (e_2) = 0. $$

Then,
$$\Delta^2  e_2^2 = 2 \Delta (\nabla e_2 \cdot \nabla e_2) = 4 ||\Hess (e_2)||^2 =  4 d_N(d_N -1). $$
Further, $\Delta e_1^2 = 2 \nabla e_1 \cdot \nabla e_1 = 2 d_N$, so that
$$\Delta^2 e_1^4 = \Delta  (2 (\Delta e_1^2) e_1^2 + 2 \nabla e_1^2 \cdot \nabla e_1^2) = \Delta( 4 \N e_1^2 + 2 e_1^2 \N) = 12 \N^2. $$

We recall Newton's identities,
$$\left\{ \begin{array}{l} e_1 = p_1 \\ \\
2 e_2 = e_1 p_1 - p_2 \\ \\
3 e_3 = e_2 p_1 - e_1 p_2 + p_3 \\ \\
4 e_4 = e_3 p_1 - e_2 p_2 + e_1 p_3 - p_4. \end{array} \right. $$

We note that $\Delta e_k \equiv 0$ for all $k$, so at $X = 0$,
\begin{equation} \label{LIST2} \left\{ \begin{array}{l} 
\Delta^2 S_{1,1,1,1}  = \Delta e_4 \equiv 0; \;\;\; \\ \\ \Delta S_{2,1,1} = 
\Delta^2 e_1 e_3 = 2 \Delta (\nabla e_1 \cdot \nabla e_3) = \nabla e_1 \cdot \nabla \Delta e_3 = 0;\\ \\
\Delta^2 S_{2,2,0} = \Delta^2 e_2^2 = 2 \Delta (\nabla e_2 \cdot \nabla e_2) = 4 || \mbox{Hess} (e_2)||^2 
= 4 d_N ( d_N -1)  \\ \\
\Delta S_{4, 0, 0} = \Delta^2 e_1^4 - (3) 8 \nabla e_1 \cdot \nabla (\nabla e_1 \cdot \nabla e_2)  + 4 ||\Hess (e_2)||^2
= 12 \N^2  + 4  d_N ( d_N -1)\\ \\
\Delta S_{3, 1, 0} =  8 \nabla e_1 \cdot \nabla (\nabla e_1 \cdot \nabla e_2) - ||\Hess (e_2)||^2 
= - 4  ||\Hess(e_2)||^2 = - 4 d_N ( d_N -1).
\end{array} \right.  \end{equation}

By  routine calculations and Lemma \ref{OV} we have,

\begin{equation}  \begin{array}{lll}
\Delta^2 \hat{\mu}_{\vec \Lambda}(0) & = & (\N-1)! \cdots 0! \sum_{\mu: |\mu| = 4}
\frac{\Delta^2 S_{\mu}(0) S_{\mu} (i \vec \Lambda)}{(\mu_1 + \N-1)! (\mu_2 + \N-2)! \cdots \mu_{\N}!}\\ &&\\
& = & (\N-1)!  (\N-2)! \frac{\Delta^2 S_{2,2,0} (0) S_{2,2, 0}(i \vec \Lambda)}{(\N + 1)! (\N )!  }
\\ &&\\
& + & (\N-1)!   \frac{\Delta^2 S_{4, 0, 0} (0) S_{4, 0, 0}(i \vec \Lambda)}{(\N + 3)! }\\
& + & (\N-1)!  (\N-2)!  \frac{\Delta^2 S_{3,1, 0} (0) S_{3, 1, 0}(i \vec \Lambda)}{(\N + 2)! (\N - 1 )!} \\ &&\\
& = &  \frac{\Delta^2 S_{2,2,0} (0) S_{2,2, 0}(i  \vec \Lambda)}{(\N + 1) \N^2 (\N-1) }
\\ &&\\
& + &    \frac{\Delta^2 S_{4, 0, 0} (0) S_{4, 0, 0}(i \vec \Lambda)}{(\N + 3) (\N+2) (\N+1)\N} +   \frac{\Delta^2 S_{3,1, 0} (0) S_{3, 1, 0}(i \vec \Lambda)}{(\N + 2)(\N+1) \N  (\N - 2)}
. \end{array} \end{equation}
By \eqref{LIST2}, we then have


\begin{equation}  \begin{array}{l}
\Delta^2 \hat{\mu}_{\vec  \Lambda}(0) 
=  \frac{4 d_N(d_N -1)  S_{2,2, 0}(i \vec \Lambda)}{(\N + 1) \N^2 (\N-1) }
+ \frac{ (12 \N^2 + 2 d_N(d_N -1)  )S_{4, 0, 0}(i \vec \Lambda)}{(\N + 3) (\N+2) (\N+1)\N} +   \frac{- 4 d_N(d_N -1)  S_{3, 1, 0}(i \vec \Lambda)}{(\N + 2)(\N+1) \N  (\N - 2)}
. \end{array} \end{equation}

Recalling \eqref{LIST}
and that $e_1(\vec \Lambda) = 0$, we get

\begin{equation}  \begin{array}{l}
\Delta^2 \hat{\mu}_{\vec \Lambda}(0) 
=  \frac{4 d_N (d_N -1) e_2^2 (i  \vec \Lambda)}{(\N + 1) \N^2 (\N-1) }
+    \frac{ (12 \N^2 + 4  d_N (d_N -1) ) e_2^2 (i \vec \Lambda)}{(\N + 3) (\N+2) (\N+1)\N} + \frac{- 4  d_N (d_N -1) e_2^2 (i \vec \Lambda)}{(\N + 2)(\N+1) \N  (\N - 2)}
. \end{array} \end{equation}

Further recalling that  $2 e_2 = e_1 p_1 - p_2$ we finally get
\begin{equation}  \begin{array}{l}
\Delta^2 \hat{\mu}_{\vec \Lambda}(0) =  \frac{ d_N(d_N -1)  p_2^2 (i \vec \Lambda)}{(\N + 1) \N^2 (\N-1) }
+    \frac{ (3 \N^2 +  d_N(d_N -1)  ) p_2^2 (i \vec \Lambda)}{(\N + 3) (\N+2) (\N+1)\N} 
+\frac{-  d_N(d_N -1)  p_2^2 (i \vec \Lambda)}{(\N + 2)(\N+1) \N  (\N - 2)}
. \end{array} \end{equation}
Since the polynomials are homogeneous of degree $4$, the factor of $i$ inside the polynomials may be removed,
and  we get  

$$\begin{array}{lll}  \E||\J (U^* D_0(\vl) U) ||^4  & = &
  \frac{ d_N(d_N -1)  p_2^2 ( \vec \Lambda)}{(\N + 1) \N^2 (\N-1) }
+    \frac{ (3 \N^2 +  d_N(d_N -1)  ) p_2^2 (\vec \Lambda)}{(\N + 3) (\N+2) (\N+1)\N} 
+\frac{-  d_N(d_N -1)  p_2^2 ( \vec \Lambda)}{(\N + 2)(\N+1) \N  (\N - 2)}
. \end{array} $$

As $N \to \infty$ the leading asymptotics of the outer terms cancel and the middle term is asymptotic to 
$ \frac{4}{d_N^2}  p_2^2(  \vec \Lambda_N)$.  We note that $\frac{p_2( \vec \Lambda_N)}{d_N}$ is bounded.
If the the empirical measure of eigenvalues tends to a limit measure, then  $\frac{p_2( \vec \Lambda_N)}{d_N}$
tends to its second moment.

Together with Lemma \ref{orbit-int}, this  completes the proof of Proposition \ref{1}. Corollary \ref{VARCOR}
follows by subtracting the square of the expectation.

\section{Completion of proof of Theorem \ref{MAINTHEO} }

By the assumption of Definition \eqref{TRACE}
\begin{equation} \omega(A)   =
\frac{1}{d_N}{\rm Tr}\;T^A_N +o(1)\;,\end{equation}  
By Lemma \ref{VARFORM}-Corollary \ref{VARCOR},  the variances of the independent random
variables $\frac{1}{d_N}Y^A_N$ are bounded. Hence, as explained in the introduction (see also \cite{Z1,SZ}),  (\ref{EP*}) follows from
 Lemma \ref{VARFORM}-Corollary \ref{VARCOR} and the Kolmogorov strong law of large numbers, which gives
\begin{equation}\label{EP*ave}\lim_{N\rightarrow \infty} \frac{1}{N}
\sum_{n=1}^N \left(\frac{1}{d_n}Y^A_n\right)=0 \;\;\; \mbox{almost surely}\;.\end{equation}
By (\ref{Y}),
$$\sup_{ \ocal \ncal \bcal_N } |X^{A}_N - \frac{1}{d_N} Y^{A}_N| =o(1) .$$ Hence also  
\begin{equation}\label{EP*aveb}\lim_{N\rightarrow \infty} \frac{1}{N}
\sum_{n=1}^N \left(\frac{1}{d_n}X^A_n\right)=0 \;\;\; \mbox{almost surely}\;.\end{equation}

If the dimensions $d_N$ grow fast enough so that $\frac{1}{d_N}$ is summable, then we obtain a stronger 
form from the fact that $\sum_{n = 1}^{\infty} \E \frac{1}{d_n} Y_n^A $ is finite hence the general term must
tend to zero almost everywhere. It follows again that $\E \frac{1}{d_n} X_n^A  \to 0$ almost everywhere.

\section{Applications}

\subsection{\label{TORUS} Fat tori}

Theorem \ref{MAINTHEO} applies to eigenspaces of the Laplacian on the flat torus $\R^d/\Z^d$ (or other
rational lattices)  of dimension $\geq 5$ and
for many eigenspaces in dimensions $d = 2, 3, 4$.

\begin{prop} \label{RATIONALTOR} Random orthornomal bases of $\Delta$-eigenspaces  of the flat torus $\R^d/\Z^d$ are quantum ergodic 
for $d \geq 5$. Also for $d = 2, 3, 4$ for special eigenspaces (specified below). \end{prop}


The only condition on the eigenspaces for Theorem \ref{MAINTHEO}   is that \eqref{TRACE} holds, and we now
recall the known results on this problem. 
 Given $A \in \Psi^0$, we denote the eigenspaces on a flat torus, enumerated in order of the eigenvalue by $\hcal_N$
and by $\Pi_N$ the orthogonal projection to $\hcal_N$  \footnote{Thanks to Z. Rudnick for explanations and
references}.

\begin{lem} \label{FLAT} The condition \eqref{TRACE} is valid in dimensions $\geq 5$ on $\R^d/\Z^d$. That is,
$$\frac{1}{d_N} Tr A \Pi_N \sim   \int_{S^* T^m} a(x, \omega) dx \wedge d \omega. $$
It follows that $\frac{1}{d_N} Y_N^A \to 0$ almost surely. 

In dimensions 2, resp. 3,  resp. 4 there are restrictions on the sequence of eigenvalues given in \cite{EH}, resp. \cite{DSP}, resp. 
\cite{P}. For eigenvalues in the allowed sequences, \eqref{TRACE} is valid.
\end{lem}

\begin{proof} We use the basis $e^N_k = e^{i \langle k, x \rangle}$ with $|k| = \mu_N$. Then
$$\langle A e^N_k, e^N_k \rangle = \int_{\R^n/Z^n} \sigma_A(x, k) dx. $$
Hence
$$\frac{1}{d_N} Tr \Pi_N A = \sum_{k: |k| = \mu_N}  \int_{\R^n/Z^n} \sigma_A (x, k) dx. $$
 In dimensions
$n \geq 5$, $d_N \sim \mu_N^{n - 2}$. It is proved that  lattice points of fixed norm on a sphere of radius $\sqrt{n}$  become uniformly distributed as $n \to \infty$ \cite{P}.  It follows that
$$\frac{1}{d_N}  \sum_{k: |k| = \mu_N}  \int_{\R^n/Z^n} \sigma_A (x, k) dx
 \to    \int_{S^* T^m} a(x, \omega) dx \wedge d \omega. $$
As in the last step of the proof of Theorem \ref{MAINTHEO},
 $\E \frac{1}{d_N} X_N^A$ is summable when $n \geq 5$.

The Liouville limit formula  is true in dimension 4 when the number of lattice points
grows linearly in $n$. The condition on n is given in \cite{P}.   In dimension 3, the equidistribution result is proved in \cite{DSP} with similar conditions
on the sequence of integers $n$.

Dimension $2$ is more complicated.  In dimension 2,  the eigenvalues of integers $n$ for
which there exist lattice points $(a,b)$ on the  circle $a^2 +b^2=n$. It is necessary that all prime factors of $n$ are congruent to 1 modulo 4. In \cite{EH} it is shown that for almost all such n, the lattice points on the circle become uniformly distributed
as $n \to \infty$.
\end{proof}

\begin{rem} In the case of a generic lattice $L \subset \R^d$, the multiplicity of eigenvalues of $\Delta$
on  $\R^d/ L$  is two. The analogue of the eigenspaces above are spectral subspaces for $\sqrt{\Delta}$ of
shriking width $w$. Thus, one considers the exponentials $e^{i \langle \ell, x \rangle}$ for $\ell \in L$
with $|\ell| \in [\lambda - C w, \lambda +w]$.  It follows from the lattice point results of  \cite{G}  that in dimensions $d \geq 5$,
the number of eigenvalues of an irrational flat torus in $[\lambda, \lambda + O(\lambda^{-1}]$ is of
order $\lambda^{d- 2}$. The question whether the  trace asymptotics  \eqref{TRACE}  hold for the
span of the corresponding eigenfunctions does not appear to have been studied.
\end{rem}

\subsection{\label{QM} Quasi-modes}

Theorem \ref{MAINTHEO} is not restricted to eigenspaces of the Laplacian and is equally valid for spaces of
quasi-modes.  We refer to \cite{CV, Po} for background on quasi-modes. Following \cite{Po},  we define  a $C^{\infty}$ quasimode of infinite order for $\hbar^2 \Delta$ with index set
${\mathcal M}_h$ to be  a family $${\mathcal Q} = \{(\psi_m(\cdot, \hbar), \mu_m(\hbar)):
m \in {\mathcal M}_{\hbar} \}$$
of  approximate eigenfunctions satisfying
\begin{equation}\left\{ \begin{array}{l} (i)  ||( \hbar^2 \Delta - \mu_m(\hbar)) \psi_m(\cdot, \hbar)||_{H^s} = O_M (\hbar^{M}),\;\;\;
(\forall M \in \Z^+), \\ \\ (ii) |\langle \psi_m, \psi_n \rangle - \delta_{mn} | =
 O_M (\hbar^{M}),\;\;\;
(\forall M \in \Z^+). \end{array} \right. \end{equation}
It follows by the spectral theorem that
for any $M \in \Z^+$, there exists at least one eigenvalue of $\hbar^2 \Delta$ in the interval
 $$I_{m M}^{\hbar} = [\mu_m(\hbar) - \hbar^M,\mu_m(\hbar) + \hbar^M],$$
and
\begin{equation} \label{QMEST} ||E_{I_{m M}^{\hbar} } \psi_k - \psi_k||_{H^s} =  O_M (\hbar^{M}). \end{equation}
Here,  $E_I$ denotes the spectral projection for $\hbar^2 \Delta$ corresponding
to the interval $I$. We 
denote the quasi-classical eigenvalue spectrum of $\hbar \sqrt{\Delta}$by
$$QSp_{\hbar} = \{\mu_m(\hbar): m \in {\mathcal M}_{\hbar} \}.$$

Since quasi-eigenvalues $\mu_m( \hbar)$ are only defined up to errors of
order $h^{\infty}$, there is  a notion of `multiple quasi-eigenvalue' defined as
follows: we say $\mu_m(\hbar) \sim \mu_n (\hbar)$ if $\mu_m - \mu_n = O(\hbar^{\infty})$
and define the multiplicity of $\mu_m(\hbar)$ by
$$mult (\mu_m(\hbar)) = \# \{n: \mu_m(\hbar) \sim \mu_n (\hbar) \} =
\dim Span \{\psi_n(\cdot, \hbar): (\hbar^2 \Delta - \mu_m(\hbar)) \psi_n = O(\hbar^{\infty}) \}. $$
We then
introduce slightly larger intervals ${\mathcal I}_{m, \hbar}$ (if need be) so that
$$QSp(\hbar) \subset \bigcup_{m \in {\mathcal M'}} {\mathcal I}_{m, \hbar},\;\;\; {\mathcal I}_{m, \hbar} \cap {\mathcal I}_{n, \hbar} = \emptyset \;\; (m \not= n).$$
Here, ${\mathcal M}'$ consists of equivalence classes of indices (corresponding
to equivalence classes of quasimodes).
We denote by ${\mathcal H}_m^{\hbar}$ the span of the quasimodes
$\{\psi_m(\cdot, \hbar): \mu_m (\hbar) \in {\mathcal
I}_m^{\hbar}\}$. Then
$$|| E_{{\mathcal I}_m^{\hbar}} v - v || = O(\hbar^{\infty}), \;\; \mbox{if}\;\; v \in {\mathcal H}_m^{\hbar}. $$

Theorem \ref{MAINTHEO} applies to quasi-mode spaces ${\mathcal H}_m^{\hbar}$ as long as their dimensions tend to infinity
and as long as there exists a unique limit state for $\frac{1}{\dim {\mathcal H}_m^{\hbar}} Tr A |_{{\mathcal H}_m^{\hbar}}$.
One might expect true modes (eigenfunctions) with eigenvalues in the intervals $I_{m M}^{\hbar}$ to be close to
linear combinations of the quasi-modes with quasi-eigenvalues in that interval.  The question raised by Theorem 
\ref{MAINTHEO} is whether they behave like random linear combinations or not. If they do, Theorem \ref{MAINTHEO}
gives their quantum limits. 

In particular, this bears on the question whether $\Delta$-eigenfunctions of 
compact  Riemannian
manifolds $(M, g)$ with KAM geodesic flow might be quantum ergodic. It seems unlikely that they are, but we are not
aware of a proof that they are not. 
For such KAM $(M, g)$,   a large family of quasi-modes is constructed in \cite{CV,Po} which localize
on the invariant tori of the KAM  Cantor set of tori.
Without reviewing the results in detail, the `large' family has positive spectral density, i.e. the number of quasi-eigenvalues
$\leq \mu$ grows like a positive  constant times $\mu^n$ where $n = \dim M$.

To our knowledge, the multiplicities and  trace  asymptotics for KAM quasi-modes  have not been studied at this time.  As
in the discussion of flat tori, one would need to determine the equidistribution law of the tori in the invariant Cantor set 
corresponding to eigenvalues (or pseudo-eigenvalues)   of $\sqrt{\Delta}$ in very short  intervals $I_{\lambda} = [\lambda - w,
\lambda + w]$.  
The orthonormal basis of eigenfunctions is not simple to relate to the near orthonormal basis of quasi-modes in this case,
but we might expect that a positive density of the eigenfunctions are mainly given as linear combinations of KAM 
quasi-modes with quasi-eigenvalues very close to the true eigenvalues. Whether or not they are quantum ergodic 
would reflect the extent to which they are sufficiently random combinations of quasi-modes and the extent to which
the collection of quasi-modes in $I^{\hbar}_{M, m}$ is Liouville distributed.  


\end{document}